\documentclass[10pt]{amsart}

\usepackage{mathrsfs}
\usepackage{amsfonts}
\usepackage{amssymb}
\usepackage{amsxtra}
\usepackage{dsfont}

\usepackage[english,polish]{babel}

\newcommand{\norm}[1]{{\left\lVert #1 \right\rVert}}
\newcommand{\abs}[1]{{\lvert {#1} \rvert}}
\newcommand{\sabs}[1]{{\left\lvert {#1} \right\rvert}}

\renewcommand{\atop}[2]{\genfrac{}{}{0pt}2{#1}{#2}}

\newcommand{\CC}{\mathbb{C}}
\newcommand{\ZZ}{\mathbb{Z}}
\newcommand{\NN}{\mathbb{N}}
\newcommand{\RR}{\mathbb{R}}
\newcommand{\PP}{\mathbb{P}}

\newcommand{\pl}[1]{\foreignlanguage{polish}{#1}}

\newtheorem{theorem}{Theorem}
\newtheorem{proposition}{Proposition}[section]
\newtheorem{lemma}{Lemma}
\newtheorem{corollary}{Corollary}

\newtheorem*{theorem*}{Theorem}

\newcommand{\ind}[1]{{\mathds{1}_{{#1}}}}
\newcommand{\seq}[2]{\left({#1}: {#2}\right)}

\title[Cotlar's ergodic theorem]
{Cotlar's ergodic theorem \\ along the prime numbers}

\author{Mariusz Mirek}
\address{Mariusz Mirek \\
	Universit\"{a}t Bonn \\
	Mathematical Institute\\
	Endenicher Allee 60\\
	D--53115 Bonn \\
	Germany \&
	Instytut Matematyczny\\
	Uniwersytet \pl{Wroc{\lll}awski}\\
	Pl. Grun\-waldzki 2/4\\
	50-384 \pl{Wroc{\lll}aw}\\
	Poland}
 \email{mirek@math.uni-bonn.de}

\author{Bartosz Trojan}
\address{
	Bartosz Trojan\\
	Instytut Matematyczny\\
	Uniwersytet \pl{Wroc{\lll}awski}\\
	Pl. Grun\-waldzki 2/4\\
	50-384 \pl{Wroc{\lll}aw}\\
	Poland}
\email{trojan@math.uni.wroc.pl}

\thanks{
	The authors were supported by NCN grant DEC--2012/05/D/ST1/00053}

\begin{document}
\selectlanguage{english}

\begin{abstract}
The aim of this paper is to prove Cotlar's ergodic theorem modeled on the set of primes.
\end{abstract}

\maketitle

\section{Introduction}
Let $(X, \mathcal{B}, \mu, S)$ be a dynamical system on a measure space $X$ endowed with
a $\sigma$-algebra $\mathcal{B}$, a $\sigma$-finite measure $\mu$ and an invertible measure
preserving transformation $S:X \rightarrow X$. In 1955 Cotlar (see \cite{cot}) established the
almost everywhere convergence of the ergodic truncated Hilbert transform
\begin{align*}
	\lim_{N\to\infty}\sum_{1 \leq \abs{n} \leq N}\frac{f(S^nx)}{n}
\end{align*}
for all $f\in L^r(\mu)$ with $1\le r<\infty$. The aim of the present paper is to obtain the
corresponding result for the set of prime numbers $\PP$. Let $\PP_N = \PP \cap (1, N]$.
We prove
\begin{theorem}
	\label{thm:1}
	For a given dynamical system $(X, \mathcal{B}, \mu, S)$ the almost everywhere convergence of
	the ergodic truncated Hilbert transform along $\PP$
	\begin{align*}
		\lim_{N\to\infty} \sum_{p\in\pm\PP_N} \frac{f(S^p x)}{p}\log \abs{p}
	\end{align*}
	holds for all $f\in L^r(\mu)$ with $1< r<\infty$.
\end{theorem}
In view of the transference principle, it is more convenient to work with the set of integers
rather than an abstract measure space $X$. In these settings we consider discrete singular
integrals with Calder\'{o}n--Zygmund kernels. Given $K \in C^1\big(\RR \setminus \{0\}\big)$
satisfying
\begin{equation}
	\label{eq:46}
	\abs{x} \abs{K(x)} + \abs{x}^2 \abs{K'(x)} \leq 1
\end{equation}
for $\abs{x} \geq 1$, together with a cancellation property
\begin{equation}
	\label{eq:47}
	\sup_{\lambda \geq 1}
	\bigg\lvert
	\int\limits_{1 \leq \abs{x} \leq \lambda} K(x) dx
	\bigg\rvert \leq 1
\end{equation}
a singular transform $T$ along the set of prime numbers is defined for a finitely
supported function $f: \ZZ \rightarrow \CC$ as
$$
T f(n) = \sum_{p \in \pm \PP} f(n - p) K(p) \log \abs{p}.
$$
Let $T_N$ denote the truncation of $T$, i.e.
$$
T_N f(n) = \sum_{p \in \pm \PP_N} f(n-p) K(p) \log \abs{p}.
$$
We show
\begin{theorem}
	\label{thm:2}
	The maximal function
	$$
	T^* f(n) = \sup_{N \in \NN} \big\rvert T_N f(n) \big\rvert
	$$
	is bounded on $\ell^r(\ZZ)$ for any $1 < r < \infty$. Moreover, the pointwise limit
	$$
	\lim_{N \to \infty} T_N f(n)
	$$
	exists and coincides with the Hilbert transform $Tf$ which is also bounded on $\ell^r(\ZZ)$
	for any $1 < r < \infty$.
\end{theorem}
For $r = 2$, the proof of Theorem \ref{thm:2} is based on the Hardy and Littlewood circle method.
These ideas were pioneered by Bourgain (see \cite{Bou1, Bou2, Bou}) in the context of pointwise
ergodic theorems along integer valued polynomials. For $r \neq 2$, initially we wanted to follow
elegant arguments from \cite{wrl} which used very specific features of the set of prime numbers.
However, we identified an issue in \cite{wrl} (see Appendix \ref{apx:1}) which made the proof
incomplete. Instead, we propose an approach (see Lemma \ref{lem:1} and \ref{lem:2}) which rectifies
Wierdl's proof (see Appendix \ref{apx:1} for details) as well as simplifies Bourgain's arguments.

Bourgain's works have inspired many authors to investigate discrete analogues of classical
operators with arithmetic features (see e.g \cite{IMSW, IW,  MSW, O, P, P1, SW0, SW1, SW2}).
Nevertheless, not many have been proved for the operators and maximal functions modelled on
the set of primes (see e.g \cite{Na1, Na2, wrl}). To the authors best knowledge, there are no other
results dealing with maximal functions corresponding with truncated discrete singular integrals.

It is worth  mentioning that Theorem \ref{thm:2} extends the result of Ionescu and Wainger
\cite{IW} to the set of prime numbers. However, our approach is different and provides a stronger
result since we study maximal functions corresponding with truncations of discrete singular
integral rather than the whole singular integral. Furthermore, we were able to define the singular
integral as a pointwise limit of its truncations. Theorem \ref{thm:2} encourages us to study
maximal functions associated with truncations of the Radon transforms from \cite{IW}. For more
details we refer the reader to the forthcoming article \cite{MT}.

Throughout the  paper, unless otherwise stated,  $C > 0$
stands for a large positive constant whose value may vary from occurrence to occurrence.
We will say that $A\lesssim B$ ($A\gtrsim B$) if there exists an absolute
constant $C>0$ such that $A\le CB$ ($A\ge CB$). If $A\lesssim B$ and $A\gtrsim B$ hold
simultaneously then we will shortly write that $A\simeq B$. We will write $A\lesssim_{\delta} B$
($A\gtrsim_{\delta} B$) to indicate that the constant $C>0$ depends on some $\delta>0$.

We always assume zero belongs to the natural numbers set $\NN$.

\section{Preliminaries}
\label{sec:2}
We start by recalling some basic facts from  number theory. A general reference is \cite{nat}.
Given $q \in \NN$ we define $A_q$ to be the set of all $a \in \ZZ \cap [1, q]$ such that
$(a, q) = 1$. By $\mu$ we denote M\"obious function, i.e. for
$q=p_1^{\alpha_1} \cdot p_2^{\alpha_2}\cdot \ldots \cdot p_n^{\alpha_n}$ where
$p_1,\ldots,p_n\in\PP$
$$
\mu(q) =
\begin{cases}
	(-1)^n & \text{if } \alpha_1 = \alpha_2 = \ldots = \alpha_n = 1,\\
	0 & \text{otherwise.}
\end{cases}
$$
In what follows, significant role will be played by the Ramanujan's identity
$$
\mu(q) = \sum_{r \in A_q}  e^{2 \pi i r a/ q} \ \ \mbox{if \  $(a, q)=1$},
$$
and the M\"obious inversion formula
\begin{align}
	\label{eqm:4}
	\sum_{a \in A_q} F(a/q)=\sum_{d \mid q}\mu(q/d)\sum_{a=1}^d F(a/d)
\end{align}
satisfied by any function $F$. Let $\varphi$ be Euler's totient function, i.e. for $q \in \NN$
the value $\varphi(q)$ is equal to the number of elements in $A_q$. Then for every $\epsilon > 0$
there is a constant $C_{\epsilon} > 0$ such that
\begin{equation}
	\label{eq:5}
	\varphi(q) \geq C_{\epsilon} q^{1-\epsilon}.
\end{equation}
Eventually, if we denote by $d(q)$ the number of divisors of $q$ then for every $\epsilon > 0$
there is a constant $C_{\epsilon} > 0$ such that
\begin{equation}
	\label{eq:6}
	d(q) \leq C_{\epsilon} q^{\epsilon}.
\end{equation}

\section{Maximal function on $\ZZ$}
\label{sec:3}
The measure space $\ZZ$ with the counting measure and
the bilateral shift operator will be our model dynamical system which permits us to prove Theorem
\ref{thm:1}.

Let us fix $\tau \in (1, 2]$ and define a set $\Lambda = \{\tau^j: j \in \NN\}$. Given a kernel
$K \in C^1(\RR \setminus \{0\})$ satisfying \eqref{eq:46} and \eqref{eq:47} we consider
a sequence $\seq{K_j}{j \in \NN}$ where
$$
K_j(x) =
\begin{cases}
	K(x) & \text{ if } \abs{x} \in (\tau^j, \tau^{j+1}],\\
	0 & \text{ otherwise.}
\end{cases}
$$
Let $\mathcal{F}$ denote the Fourier transform on $\RR$ defined for any function $f \in L^1(\RR)$
as
$$
\mathcal{F} f(\xi) = \int_\RR f(x) e^{2\pi i \xi x} dx.
$$
If $f \in \ell^1(\ZZ)$ we set
$$
\hat{f}(\xi) = \sum_{n \in \ZZ} f(n) e^{2\pi i \xi n}.
$$
Then for $\Phi_j = \mathcal{F} K_j$ by integration by parts one can show
\begin{equation}
	\label{eq:39}
	\abs{\Phi_j(\xi)} \lesssim \abs{\xi}^{-1} \tau^{-j}.
\end{equation}
We define a sequence $\seq{m_j}{j \in \NN}$ of multipliers
$$
m_j(\xi) = \sum_{p \in \pm \PP} e^{2\pi i \xi p} K_j(p) \log \abs{p}.
$$

\subsection{$\ell^2$-approximation}
\label{subsec:3}
To approximate the multiplier $m_j$ we adopt the argument introduced by Bourgain \cite{Bou}
(see also Wierdl \cite{wrl}) which is based on the Hardy--Littlewood circle method (see e.g
\cite{vau}).

For any $\alpha > 0$ and $j \in \NN$ major arcs are defined by
$$
\mathfrak{M}_j = \bigcup_{1\le q \leq j^\alpha} \bigcup_{a \in A_q} \mathfrak{M}_j(a/q)
$$
where
$$
\mathfrak{M}_j(a/q) =
\big\{
	\xi \in [0, 1]:
	\abs{\xi - a/q} \leq \tau^{-j} j^\alpha
\big\}.
$$
Here and subsequently we will treat the interval $[0, 1]$ as the circle group $\Pi=\RR/\ZZ$
identifying $0$ and $1$.
\begin{proposition}
	\label{prop:1}
	For $\xi \in \mathfrak{M}_j(a/q) \cap \mathfrak{M}_j$
	$$
	\Big\lvert
	m_j(\xi) - \frac{\mu(q)}{\varphi(q)} \Phi_j(\xi - a/q)
	\Big\rvert
	\leq C_\alpha j^{-\alpha}.
	$$
	The constant $C_\alpha$ depends only on $\alpha$.
\end{proposition}
\begin{proof}
	Since for a prime number $p$, $p \mid q$ if and only if $(p\ \mathrm{mod}\ q, q) > 1$ we have
	$$
	\Big\lvert
	\sum_{\atop{1 \leq r \leq q}{(r,q) > 1}} \sum_{\atop{p \in \PP}{q \mid (p -r)}}
	e^{2\pi i \xi p} K_j(p) \log p
	\Big\rvert
	\leq
	\tau^{-j+1} \sum_{\atop{p \in \PP}{p \mid q}} \log p \lesssim \tau^{-j} \log j.
	$$
	Let $\theta = \xi - a/q$. If $p \equiv r \pmod q$ then
	$$
	\xi p \equiv \theta p + ra/q  \pmod 1
	$$
	and consequently
	\begin{equation}
		\label{eq:2}
		\sum_{r \in A_q}
		\sum_{\atop{p \in \PP}{q \mid (p - r)}}
		e^{2\pi i \xi p} K_j(p) \log p
		=
		\sum_{r \in A_q} e^{2 \pi i r a/q}
		\sum_{\atop{p \in \PP}{q \mid (p - r)}} e^{2\pi i \theta p} K_j(p) \log p.
	\end{equation}
	Using the summation by parts (see e.g \cite[p. 304]{nat}) for the inner sum on the right hand
	side in \eqref{eq:2} we obtain
	\begin{multline}
		\label{eqm:6}
		\sum_{\atop{n \in N_j}{q \mid (n-r)}} e^{2 \pi i \theta n} K(n) \ind{\PP}(n) \log n
		=
		\psi(\tau^{j+1}; q, r)e^{2 \pi i \theta \tau^{j+1}} K(\tau^{j+1})
		-\psi(\tau^{j}; q, r)e^{2 \pi i \theta \tau^{j}} K(\tau^{j})\\
		-\int_{\tau^j}^{\tau^{j+1}} \psi(t; q, r)
		\frac{d}{dt} \left(e^{2\pi i \theta t}K(t) \right) dt
	\end{multline}
	where $N_j = \NN \cap (\tau^j, \tau^{j+1}]$ and for $x \geq 2$ we have set
	$$
	\psi(x; q, r) = \sum_{\atop{p \in \PP_x}{q \mid (p-r)}} \log p.
	$$
	Similar reasoning gives
	\begin{equation}
		\label{eqm:7}
		\sum_{n \in N_j} e^{2\pi i \theta n} K(n)
		=\tau^{j+1}e^{2 \pi i \theta \tau^{j+1}} K(\tau^{j+1})
		-
		\tau^{j}e^{2 \pi i \theta \tau^{j}} K(\tau^{j})
		-
		 \int_{\tau^j}^{\tau^{j+1}} t \frac{d}{dt}
		\left( e^{2\pi i \theta t}K(t)\right)dt.
	\end{equation}
	By Siegel--Walfisz theorem (see \cite{sieg, wal}) we know that for every $\alpha>0$ and
	$x \geq 2$
	\begin{equation}
		\label{eq:34}
		\bigg\lvert \psi(x; q, r) - \frac{x}{\varphi(q)} \bigg\rvert
		\lesssim x (\log x)^{-3\alpha}
	\end{equation}
	where the implied constant depends only on $\alpha$. Therefore \eqref{eqm:6} and \eqref{eqm:7}
	combined with the estimates \eqref{eq:46} and \eqref{eq:34} yield
	\begin{multline*}
		\bigg\lvert
		\sum_{\atop{p \in \PP}{q \mid (p -r)}} e^{2\pi i\theta p} K_j(p) \log p
		-\frac{1}{\varphi(q)} \sum_{n \in \NN} e^{2\pi i \theta n} K_j(n)
		\bigg\rvert
		\lesssim
		\bigg\lvert
		\psi(\tau^{j+1};q,r) - \frac{\tau^{j+1}}{\varphi(q)}
		\bigg\rvert |K(\tau^{j+1})|\\
		+\bigg\lvert
		\psi(\tau^{j};q,r) - \frac{\tau^{j}}{\varphi(q)}
		\bigg\rvert |K(\tau^{j})|
		+
		\int_{\tau^j}^{\tau^{j+1}}
		\bigg\lvert \psi(t;q,r) - \frac{t}{\varphi(q)} \bigg\rvert
		\big(t^{-1} \abs{\theta}  + t^{-2}\big) dt\\
        \lesssim j^{-3\alpha}+
		\int_{\tau^j}^{\tau^{j+1}} (\log t)^{-3\alpha} \big(
		\abs{\theta}  + t^{-1}\big) dt
		\lesssim j^{-2\alpha}.
	\end{multline*}
	Eventually, by \eqref{eq:2},
	\begin{multline*}
		\bigg\lvert
		\sum_{r \in A_q} \sum_{\atop{p \in \PP}{q \mid (p -r)}}
		e^{2 \pi i \xi p} K_j(p) \log p
		-\frac{\mu(q)}{\varphi(q)} \sum_{n \in \NN} e^{2\pi i \theta n} K_j(n)
		\bigg\rvert	\\
		=\bigg\lvert
		\sum_{r \in A_q} e^{2 \pi i ra/q}
		\bigg(
		\sum_{\atop{p \in \PP}{q \mid (p -r)}}
		e^{2 \pi i \theta p}K_j(p) \log p
		- \frac{1}{\varphi(q)} \sum_{n \in \NN} e^{2\pi i \theta n} K_j(n)
		\bigg)
		\bigg\rvert
		\lesssim q j^{-2\alpha} \leq j^{-\alpha}.
	\end{multline*}
	Next, we can substitute an integral for the sum since for $n_0 = \lceil \tau^j \rceil$
	and $n_1 = \lfloor \tau^{j+1} \rfloor$ we have
	$$
	\int_{\tau^j}^{\tau^{j+1}} e^{2\pi i \theta t} K(t) dt =
	\int_{\tau^j}^{n_0} e^{2\pi i \theta t} K(t) dt
	+\sum_{n = n_0}^{n_1-1} \int_0^1 e^{2\pi i \theta (n + t)} K(n+t) dt
	+\int_{n_1}^{\tau^{j+1}} e^{2\pi i \theta t} K(t) dt
	$$
	thus
	\begin{multline*}
		\bigg\lvert
		\sum_{n = n_0}^{n_1-1}
		e^{2\pi i \theta n} K(n) -
		\int_0^1 e^{2\pi i \theta (n+t)} K(n+t) dt
		\bigg\rvert\\
		\leq
		\sum_{n=n_0}^{n_1-1}\int_0^1 \sabs{1-e^{-2\pi i \theta t}}\sabs{K(n)} dt
		+
		\sum_{n=n_0}^{n_1-1} \int_0^1 \sabs{K(n) - K(n+t)} dt
		\lesssim \tau^{-j} j^\alpha.
	\end{multline*}
	Repeating all the steps with $p$ replaced by $-p$ we finish the proof.
\end{proof}

For $s \in \NN$ we set
$$
\mathscr{R}_s =
\big\{
	a/q \in [0, 1]\cap\mathbb{Q}:
	2^s \leq q < 2^{s+1} \text{ and } (a, q) =1
\big\}.
$$
Since we treat $[0, 1]$ as the circle group identifying $0$ and $1$ we see that
$\mathscr{R}_0=\{1\}$. Let us consider
\begin{equation}
	\label{eq:12}
	\nu_j^s(\xi) = \sum_{a/q \in \mathscr{R}_s} \frac{\mu(q)}{\varphi(q)}
	\Phi_j(\xi - a/q) \eta_s(\xi-a/q)
\end{equation}
where $\eta_s(\xi) = \eta(A^{s+1} \xi)$ and $\eta: \RR \rightarrow \RR$ is a smooth
function such that $0 \leq \eta(x) \leq 1$ and
$$
\eta(x) = \begin{cases}
	1 & \text{for } \abs{x} \leq 1/4,\\
	0 & \text{for } \abs{x} \geq 1/2.
\end{cases}
$$
The value of $A$ is chosen to satisfy \eqref{eq:18}. Additionally, we may assume
(this will be important in the sequel) that $\eta$ is a convolution of two smooth functions
with compact supports contained in $[-1/2, 1/2]$. Let $\nu_j = \sum_{s \in \NN} \nu_j^s$.


\begin{proposition}
	\label{prop:2}
	For every $\alpha > 16$
	$$
	\big\lvert m_j(\xi) - \nu_j(\xi) \big\rvert \leq C_\alpha j^{-\alpha/4}.
	$$
	The constant $C_\alpha$ depends only on $\alpha$.
\end{proposition}
\begin{proof}
	First of all notice that for a fixed $s \in \NN$ and $\xi \in [0, 1]$ the sum \eqref{eq:12}
	consists of the single term. Otherwise, there would be $a/q, a'/q' \in \mathscr{R}_s$ such that
	$\eta_s(\xi - a/q) \neq 0$ and $\eta_s(\xi-a'/q') \neq 0$. Therefore,
	$$
	2^{-2 s - 2} \leq \frac{1}{qq'} \leq
	\Big\lvert \frac{a}{q}-\frac{a'}{q'} \Big\rvert
	\leq
	\Big\lvert \xi - \frac{a}{q} \Big\rvert +
	\Big\lvert \xi - \frac{a'}{q'} \Big\rvert
	\leq A^{-s-1}
	$$
	which is not possible whenever $A > 4$, as it was assumed in \eqref{eq:18}.

	\vspace{2ex}
	\noindent{\bf{Major arcs estimates}: $\xi\in\mathfrak{M}_j(a/q) \cap \mathfrak{M}_j$.}
	Let $s_0$ be such that
	\begin{equation}
		\label{eq:35}
		2^{s_0} \leq q < 2^{s_0+1}.
	\end{equation}
	Next, we choose $s_1$ satisfying
	$$
	2^{s_1+1} \leq \tau^j j^{-2\alpha} < 2^{s_1+2}.
	$$
	If $s < s_1$ then for any $a'/q' \in \mathscr{R}_s$, $a'/q' \neq a/q$ we have
	$$
	\Big\lvert 	\xi - \frac{a'}{q'} \Big\rvert
	\geq
	\frac{1}{qq'} - \Big\lvert \xi - \frac{a}{q} \Big\rvert
	\geq
	2^{-s-1} j^{-\alpha} - \tau^{-j} j^\alpha
	\geq
	\tau^{-j} j^\alpha.
	$$
	Therefore, the integration by parts gives
	$$
	\abs{\Phi_j(\xi - a'/q')} \lesssim (\abs{\xi - a'/q'} \tau^j)^{-1} \lesssim j^{-\alpha}.
	$$
	Combining the last estimate with \eqref{eq:5}, we obtain that for some $\delta'>0$
	$$
	I_1=\bigg\lvert
	\sum_{s = 0}^{s_1-1} \sum_{\atop{a'/q' \in \mathscr{R}_s}{a'/q' \neq a/q}}
	\frac{\mu(q')}{\varphi(q')} \Phi_j(\xi - a'/q') \eta_s(\xi - a'/q')
	\bigg\rvert
	\lesssim
	j^{-\alpha}
	\sum_{s=0}^{s_1-1}
	2^{-\delta' s}.
	$$
	Moreover, if $\eta_{s_0}(\xi -a/q) < 1$ then $\abs{\xi - a/q} \geq 4^{-1} A^{-s_0-1}$. By
	\eqref{eq:35} we have $2^{s_0} \leq j^\alpha$. Hence the integration by parts implies
	$$
	I_2=\bigg\lvert
	\frac{\mu(q)}{\varphi(q)} \Phi_j(\xi - a/q) \big(1 - \eta_{s_0}(\xi-a/q)\big)
	\bigg\rvert
	\lesssim
	A^{s_0+1} \tau^{-j}
	\lesssim
	j^{-\alpha}.
	$$
	In the last estimate it is important that the implied constant does not depend on $s_0$.
	Since $\Phi_j$ is bounded uniformly with respect to $j \in \NN$, by \eqref{eq:5} and the
	definition of $s_1$ we have
	$$
	I_3=\bigg\lvert
	\sum_{s = s_1}^\infty\sum_{\atop{a'/q' \in \mathscr{R}_s}{a'/q' \neq a/q}}
	\frac{\mu(q')}{\varphi(q')} \Phi_j (\xi - a'/q') \eta_s(\xi - a'/q')
	\bigg\rvert
	\lesssim
	\sum_{s = s_1}^\infty 2^{-\delta'' s}
	\lesssim (\tau^{-j}j^{2\alpha})^{\delta''}\lesssim j^{-\alpha}
	$$
	for appropriately chosen $\delta''>0$. Eventually, in view of Proposition \ref{prop:1} and
	definitions of $s_0$ and $s_1$ we conclude
	\begin{align*}
		\big\lvert m_j (\xi) - \nu_j(\xi) \big\rvert
		\leq C_\alpha j^{-\alpha}+ I_1+I_2+I_3\lesssim j^{-\alpha}.
	\end{align*}
	
	\vspace*{2ex}
	\noindent{\bf Minor arcs estimates: $\xi\not\in\mathfrak{M}_j$.} Firstly, by the summation by
	parts, we get
	\begin{equation}
		\label{eq:27}
		\Big\lvert
		\sum_{p \in \PP} e^{2\pi i \xi p} K_j (p) \log p
		\Big\rvert \leq \abs{F_{\tau^{j+1}}(\xi)}|K(\tau^{j+1})|
		+\abs{F_{\tau^{j}}(\xi)}|K(\tau^{j})|+
		\int_{\tau^j}^{\tau^{j+1}} \abs{F_t(\xi)}
		\lvert
		K'(t)
		\rvert
		dt
	\end{equation}
	where
	$$
	F_x(\xi) = \sum_{p \in \PP_x} e^{2 \pi i \xi p} \log p.
	$$
	Using Dirichlet's principle there are $(a, q) = 1$, $j^\alpha \leq q \leq \tau^j j^{-\alpha}$
	such that
	$$
	\abs{\xi - a/q} \leq q^{-1} \tau^{-j} j^\alpha \leq q^{-2}.
	$$
	Thus, by Vinogradov's theorem (see \cite[Theorem 1, Chapter IX]{vin} or \cite[Theorem 8.5]{nat})
	we get
	$$
	\abs{F_t(\xi)} \lesssim j^4 (\tau^j q^{-1/2} + \tau^{4 j/5} + \tau^{j/2} q^{1/2})
	\lesssim \tau^j j^{4 - \alpha/2}
	$$
	for $t \in [\tau^j, \tau^{j+1}]$. Combining $\abs{K'(t)} \lesssim \tau^{-2j}$ with the
	last bound and \eqref{eq:27} we conclude
	$$
	\abs{m_j(\xi)} \lesssim j^{4 - \alpha/2}\lesssim j^{-\alpha/4}
	$$
	since $\alpha > 16$. In order to estimate the $\nu_j$ let us define $s_1$ by setting
	$$
	2^{s_1} \leq j^{\alpha/2} < 2^{s_1+1}.
	$$
	If $a/q \in \mathscr{R}_s$ for $s < s_1$ then $q < j^{\alpha}$ and
	$$
	\Big \lvert \xi - \frac{a}{q} \Big\rvert
	\geq 2^{-s - 1} \tau^{-j} j^{\alpha} \gtrsim \tau^{-j} j^{\alpha/2}.
	$$
	Then again by the integration by parts we obtain
	$$
	\abs{\Phi_j(\xi - a/q)} \lesssim (\abs{\xi - a/q} \tau^j)^{-1} \lesssim j^{-\alpha/2}.
	$$
	Therefore, the first part of the sum may be majorized by
	$$
	\Big\lvert \sum_{s = 0}^{s_1-1} \nu_j^s(\xi) \Big\vert
	\lesssim j^{-\alpha/2} \sum_{s = 0}^{\infty} 2^{-\delta' s},
	$$
	as for $I_1$. For the second part we proceed as for $I_3$ to get
	$$
	\Big\lvert \sum_{s = s_1}^\infty \nu_j^s(\xi)\Big\rvert \lesssim \sum_{s=s_1}^\infty
	2^{-\delta'' s} \lesssim j^{-\delta'' \alpha/2}\lesssim j^{-\alpha/4}.
	$$
	A suitable choice of $\delta', \delta''>0$ in both estimates above was possible thanks to
	\eqref{eq:5}.
\end{proof}

\subsection{$\ell^r$-theory}
\label{subsec:4}
We start the section by proving two lemmas which will play crucial role.
\begin{lemma}
	\label{lem:1}
	There is a constant $C > 0$ such that for all $s \in \NN$  and $u\in\RR$
	\begin{equation}
		\label{eq:13}
		\bigg\lVert
		\int_{-\frac{1}{2}}^{\frac{1}{2}} e^{-2\pi i \xi j} \eta_s(\xi) d\xi
		\bigg\rVert_{\ell^1(j)} \leq C,
	\end{equation}
	\begin{equation}
		\label{eq:14}
		\bigg\lVert
		\int_{-\frac{1}{2}}^{\frac{1}{2}} e^{-2\pi i \xi j}
		\big(1-e^{2\pi i \xi u}\big)\eta_s(\xi)d\xi
		\bigg\rVert_{\ell^1(j)} \leq C |u| A^{-s-1}.
	\end{equation}
\end{lemma}
\begin{proof}
	We only show \eqref{eq:14} for $u\in\RR$, since the proof of \eqref{eq:13} is almost identical.
	Recall, $\eta = \phi * \psi$ for $\psi, \phi$ smooth functions with supports inside
	$[-1/2, 1/2]$. Hence, $\eta_s = A^{s+1} \phi_s * \psi_s$ and
	$$
	A^{-s-1} \int_{-\frac{1}{2}}^{\frac{1}{2}} e^{-2\pi i \xi j}\big(1 - e^{2\pi i \xi u}\big)
		\eta_s(\xi) d\xi
	=\mathcal{F}^{-1} \phi_s (j) \mathcal{F}^{-1} \psi_s(j)
	- \mathcal{F}^{-1} \phi_s(j-u) \mathcal{F}^{-1} \psi_s(j-u).
	$$
	By Cauchy--Schwatz's inequality and Plancherel's theorem
	\begin{multline*}
		\sum_{j \in \ZZ}
		\sabs{\mathcal{F}^{-1} \phi_s(j)}
		\sabs{\mathcal{F}^{-1} \psi_s(j)-\mathcal{F}^{-1} \psi_s(j-u)}
		\leq
		\norm{\mathcal{F}^{-1} \phi_s}_{\ell^2}
		\bigg\lVert
		\int_\RR e^{-2\pi i \xi j} \big(1-e^{2\pi i \xi u}\big)
		\psi_s(\xi) d\xi
		\bigg\rVert_{\ell^2(j)}\\
		=
		\norm{\phi_s}_{L^2}
		\norm{\big(1-e^{2\pi i \xi u}\big) \psi_s(\xi)}_{L^2(d\xi)}.
	\end{multline*}
	Moreover, since
	$$
	\int_\RR \sabs{1 - e^{-2\pi i \xi u}}^2 \abs{\psi_s(\xi)}^2 d\xi
	\lesssim u^2 \int_\RR \abs{\xi}^2 \abs{\psi_s(\xi)}^2 d\xi
	\lesssim u^2 A^{-3(s+1)} \norm{\psi}_{L^2}^2
	$$
	we obtain
	$$
	\sum_{j \in \ZZ} \sabs{\mathcal{F}^{-1} \phi_s(j)}
	\sabs{\mathcal{F}^{-1} \psi_s(j)- \mathcal{F}^{-1} \psi_s(j-u)}
	\lesssim |u| A^{-2(s+1)} \norm{\phi}_{L^2} \norm{\psi}_{L^2}
	$$
	which finishes the proof of \eqref{eq:14}.
\end{proof}

\begin{lemma}
	\label{lem:2}
	Let $r > 1$. For all $q \in [2^s, 2^{s+1})$, $s \ge r$ and
	$l \in \{1, 2, \ldots, q\}$
	$$
	\big\lVert
	\mathcal{F}^{-1} \big(\eta_s \hat{f}\big)(q j + l)
	\big\rVert_{\ell^{r}(j)}
	\simeq
	q^{-1/r}
	\big\lVert
	\mathcal{F}^{-1} \big(\eta_s \hat{f}\big)
	\big\rVert_{\ell^{r}}.
	$$
\end{lemma}
\begin{proof}
	We define a sequence $\big(J_1, J_2, \ldots, J_q\big)$ by
	$$
	J_l =
	\big\lVert
	\mathcal{F}^{-1} \big(\eta_s \hat{f} \big)(qj+l)
	\big\rVert_{\ell^r(j)}.
	$$
	Then $J_1^r + J_2^r + \ldots + J_q^r = I^r$ where
	$I = \big\lVert \mathcal{F}^{-1} \big(\eta_s \hat{f} \big) \big\rVert_{\ell^r(j)} $.
	Since $\eta_s = \eta_s \eta_{s-1}$, by Minkowski's inequality we obtain
	\begin{multline*}
		\big\lVert
		\mathcal{F}^{-1} \big(\eta_s \hat{f} \big)(qj+l)
		-\mathcal{F}^{-1}\big(\eta_s \hat{f} \big)(qj+l')
		\big\rVert_{\ell^r(j)}
		=
		\bigg\lVert
		\int_{-\frac{1}{2}}^{\frac{1}{2}} e^{-2\pi i \xi (q j+l)}
		\big(1 - e^{2\pi i\xi (l-l')}\big)
		\eta_s(\xi) \hat{f}(\xi) d\xi
		\bigg\rVert_{\ell^r(j)}\\
		\leq
		\bigg\lVert
		\int_{-\frac{1}{2}}^{\frac{1}{2}} e^{-2\pi i \xi j} \big(1 - e^{2\pi i\xi (l-l')}\big)
		\eta_{s-1}(\xi) d\xi
		\bigg\rVert_{\ell^1(j)} I
		\leq C q A^{-s} I
	\end{multline*}	
	where in the last step we have used Lemma \ref{lem:1}. We notice, the constant $C > 0$ depends
	only on $\eta$. Hence, for all $l,l' \in \{1,2,\ldots, q\}$
	$$
	J_l \leq J_{l'} + C q A^{-s} I.
	$$
	Since $q < 2^{s+1}$ taking
	\begin{equation}
		\label{eq:18}
		A > 32 \max\{1, C\}
	\end{equation}
	we obtain $C q A^{-s} \leq 2^{-4s+1}$ thus
	\begin{equation}
		\label{eq:15}
		J_l^r \leq 2^{r-1}J_{l'}^r + 2^{r-1}\big(C q A^{-s}\big)^r I^r
		\leq 2^{r-1}J_{l'}^r + 2^{2r-4s-1} I^r.
	\end{equation}
	Therefore,
	$$
	I^r = J_1^r + J_2^r + \ldots + J_q^r
	\leq 2^{r-1}q J_l^r + q2^{2r-4s-1} I^r\leq 2^{r-1}q J_l^r + 2^{3r-3s-1} I^r
	$$
	and using $s > r$, we get $I^r \leq 2^{r} q J_l^r$. For the converse inequality, we use
	again \eqref{eq:15} to conclude
	$$
	q J_l^r \leq 2^{r-1}\big(J_1^r + J_2^r + \ldots + J_q^r\big) + q 2^{2r-4s-1} I^r \leq 2^r I^r.
	$$
\end{proof}

\begin{proposition}
	\label{prop:3}
	For $r > 1$ and $s\in\NN$
	$$
	\Big\lVert
	\sup_{k \in \NN}
	\big\lvert
	\mathcal{F}^{-1} \big(\Psi_k \eta_s \hat{f}  \big)
	\big\rvert
	\Big\rVert_{\ell^r}
	\leq
	C_r
	\big\|\mathcal{F}^{-1} \big(\eta_s \hat{f}  \big)\big\|_{\ell^r}
	$$
	where $\Psi_k = \sum_{j = 0}^k \Phi_j$.
\end{proposition}
\begin{proof}
	Since $\eta_s=\eta_{s-1}\eta_s$ thus by H\"older's inequality we have
	\begin{multline*}
		\sup_{k \in \NN} \big\lvert \mathcal{F}^{-1}
		\big(\Psi_k \eta_s \hat{f} \big) (m) \big\rvert^r
		\leq
		\bigg(\int_{\RR}
		\sup_{k \in \NN}
		\big\lvert
		\mathcal{F}^{-1} \big(\Psi_k \eta_s\hat{f} \big)(t)\big\rvert
		\big\lvert \mathcal{F}^{-1}\eta_{s-1}(m - t)\big\rvert dt\bigg)^r\\
		\leq
		\int_{\RR} \sup_{k \in \NN}
		\big\lvert \mathcal{F}^{-1} \big(\Psi_k \eta_s\hat{f} \big)(t) \big\rvert^r
		\big\lvert \mathcal{F}^{-1} \eta_{s-1}(m - t) \big\rvert dt\
		\norm{\mathcal{F}^{-1} \eta_{s-1}}_{L^{1}}^{r-1}.
	\end{multline*}	
	Now we note that $\norm{\mathcal{F}^{-1} \eta_{s-1}}_{L^{1}}\lesssim 1$ and
	$$
	\sum_{m \in \ZZ} \big\lvert \mathcal{F}^{-1}\eta_{s-1} (m-t) \big\rvert
	\lesssim A^{-s}\sum_{m \in \ZZ} \frac{1}{1 + (A^{-s}(m - t))^2}
	\lesssim A^{-s}\bigg(1+\int_{\RR}
	\frac{dx}{1 +  (A^{-s}x)^2 }\bigg)\lesssim A^{-s}(1+A^s)\lesssim1
	$$
	and the implied constants are independent of $A$. Thus we obtain
	\begin{equation}
		\label{eq:57}
		\Big\lVert
		\sup_{k \in \NN} \big\lvert \mathcal{F}^{-1} \big(\Psi_k \eta_s\hat{f} \big)\big\rvert
		\Big\rVert_{\ell^r}
		\lesssim
		\Big\lVert
		\sup_{k \in \NN} \big\lvert \mathcal{F}^{-1} \big(\Psi_k \eta_s\hat{f} \big)\big\rvert
		\Big\rVert_{L^r}\lesssim \big\|\mathcal{F}^{-1} \big(\eta_s\hat{f}\big)\big\|_{L^r},
	\end{equation}
	where the last inequality is a consequence of \cite{riv}. The proof will be completed if we
	show
	$$
	\big\|\mathcal{F}^{-1} \big(\eta_s\hat{f}\big)
	\big\|_{L^r}\lesssim \big\|\mathcal{F}^{-1} \big(\eta_s\hat{f}\big)\big\|_{\ell^r}.
	$$
	For this purpose we use \eqref{eq:14} from Lemma \ref{lem:2}. Indeed,
	\begin{multline*}
		\big\|\mathcal{F}^{-1} \big(\eta_s\hat{f}\big)\big\|_{L^r}^r
		=
		\sum_{j\in\ZZ}\int_{0}^{1} \big|\mathcal{F}^{-1} \big(\eta_s\hat{f}\big)(x+j)\big|^rdx\\
		\leq
		2^{r-1}\big\|\mathcal{F}^{-1} \big(\eta_s\hat{f}\big)\big\|_{\ell^r}^r
		+ 2^{r-1}\sum_{j\in\ZZ}\int_{0}^{1} \big|\mathcal{F}^{-1}
		\big(\eta_s\hat{f}\big)(x+j)-\mathcal{F}^{-1} \big(\eta_s\hat{f}\big)(j)\big|^rdx\\
		=
		2^{r-1}\big\|\mathcal{F}^{-1} \big(\eta_s\hat{f}\big)\big\|_{\ell^r}^r
		+2^{r-1}\int_{0}^{1} \bigg\|\int_{-\frac{1}{2}}^{\frac{1}{2}} e^{-2\pi i \xi j}
		\big(1-e^{-2\pi i \xi x}\big)\eta_s(\xi)\hat{f}(\xi)d\xi \bigg\|_{\ell^r(j)}^rdx\\
		\leq
		2^{r-1}\big\|\mathcal{F}^{-1} \big(\eta_s\hat{f}\big)\big\|_{\ell^r}^r
		+2^{r-1}\int_{0}^{1} \bigg\|\int_{-\frac{1}{2}}^{\frac{1}{2}} e^{-2\pi i \xi j}
		\big(1-e^{-2\pi i \xi x}\big)\eta_{s-1}(\xi)d\xi\bigg\|_{\ell^1(j)}^r
		\big\|\mathcal{F}^{-1} \big(\eta_s\hat{f}\big)\big\|_{\ell^r}^rdx\\
		\lesssim
		\big\|\mathcal{F}^{-1} \big(\eta_s\hat{f}\big)\big\|_{\ell^r}^r.
	\end{multline*}
	This finishes the proof of the proposition.
\end{proof}

\begin{theorem}
	\label{th:3}
	For each $r > 1$ there are $\delta_r > 0$ and $C_r > 0$ such that
	$$
	\Big\lVert
	\sup_{k \in \NN}
	\big\lvert
	\sum_{j = 0}^k \mathcal{F}^{-1} (\nu^s_j \hat{f})
	\big\rvert
	\Big\rVert_{\ell^r}
	\leq
	C_r 2^{-\delta_r s} \norm{f}_{\ell^r}
	$$
	for all $f \in \ell^r(\ZZ)$.
\end{theorem}
\begin{proof}
	Based on Proposition \ref{prop:3} we may assume $s \geq r$. Let $q \in [2^s, 2^{s+1})$ be
	fixed. Firstly, we are going to show that for every
	$\epsilon>0$ we have
	\begin{align}
	\label{eqm:12}
		\bigg\lVert
		\sup_{k \in \NN}
		\Big\lvert
		\sum_{a \in A_q}
		\mathcal{F}^{-1} \big(\Psi_k(\cdot-a/q) \eta_s(\cdot-a/q) \hat{f}\big)
		\Big\rvert
	\bigg\rVert_{\ell^r}\leq C_{\epsilon} q^{\epsilon}\|f\|_{\ell^r}.
	\end{align}
	By M\"obius inversion formula \eqref{eqm:4} we see that
	\begin{equation}
		\label{eqm:13}
		\sum_{a \in A_q} \mathcal{F}^{-1}
		\big(\Psi_k(\cdot-a/q) \eta_s(\cdot-a/q) \hat{f}\big)(x)
		=\sum_{b \mid q} \mu(q/b) \sum_{a=1}^b e^{-2\pi i ax/b} \mathcal{F}^{-1}
		\big(\Psi_k \eta_s \hat{f}(\cdot+a/b)\big)(x).
	 \end{equation}
	Moreover, for $x \equiv l \pmod{q}$ we may write
	\begin{equation}
		\label{eq:11}
		\sum_{a=1}^b e^{-2\pi i a x/b}
		\mathcal{F}^{-1}\big(\Psi_k \eta_s \hat{f} (\cdot + a/b)\big)(x)
		=\mathcal{F}^{-1}\big(\Psi_k \eta_s F_b(\cdot\ ; l)\big) (x)
	\end{equation}
	where for $b \mid q$ we have set
	$$
	F_b(\xi; l) = \sum_{a=1}^b \hat{f}(\xi + a/b) e^{-2\pi i l a/b}.
	$$
	Therefore, by formula \eqref{eqm:13} and \eqref{eq:11} we have
	$$
	\bigg\lVert
	\sup_{k \in \NN}
	\Big\lvert
	\sum_{a \in A_q}
	\mathcal{F}^{-1} \big(\Psi_k(\cdot-a/q) \eta_s(\cdot-a/q) \hat{f}\big)
	\Big\rvert
	\bigg\rVert_{\ell^r}
	\leq
	\sum_{b \mid q}
	\bigg(
	\sum_{l=1}^q
	\Big\lVert
	\sup_{k \in \NN}
	\big\lvert
	\mathcal{F}^{-1} \big(\Psi_k \eta_s F_b(\cdot\ ; l) \big) (q j + l)
	\big\rvert
	\Big\rVert_{\ell^r(j)}^r
	\bigg)^{1/r}.
	$$
	Thus in view of \eqref{eq:6} it will suffice to prove that
	\begin{align}
		\label{eqm:14}
	  	\bigg(
		\sum_{l=1}^q
		\Big\lVert
		\sup_{k \in \NN}
		\big\lvert
		\mathcal{F}^{-1} \big(\Psi_k \eta_s F_b(\cdot\ ; l) \big) (q j + l)
		\big\rvert
		\Big\rVert_{\ell^r(j)}^r
		\bigg)^{1/r}
		\leq C_r \norm{f}_{\ell^r}
	\end{align}
	where the constant does not depend on $b$. For the proof let us fix $f \in \ell^r(\ZZ)$
	and consider a sequence $(J_1, J_2, \ldots, J_q)$ defined by
	$$
	J_l =
	\Big\lVert
	\sup_{k \in \NN} \big\lvert \mathcal{F}^{-1}
	\big(\Psi_k \eta_s \hat{f} \big)(q j + l) \big\rvert
	\Big\rVert_{\ell^r(j)}.
	$$
	By Proposition \ref{prop:3}, we have
	$$
	J_1^r + J_2^r + \ldots + J_q^r =I^r=
	\Big\lVert
	\sup_{k \in \NN}
	\big\lvert \mathcal{F}^{-1} \big(\Psi_k  \eta_s \hat{f} \big) \big \rvert
	\Big\rVert_{\ell^r}^r
	\lesssim
	\big\lVert
	\mathcal{F}^{-1} \big(\eta_s \hat{f}\big)
	\big\rVert_{\ell^r}^r.
	$$
	Also for any $l, l' \in \{1,2,\ldots,q\}$
	\begin{multline*}
		\bigg\lVert
		\sup_{k \in \NN}
		\bigg\lvert
		\int_{-\frac{1}{2}}^{\frac{1}{2}} e^{-2\pi i \xi (q j + l)}
		\big(1 - e^{2\pi i \xi (l-l')}\big) \Psi_k(\xi)\eta_s(\xi) \hat{f}(\xi) d\xi
		\bigg\rvert
		\bigg\rVert_{\ell^r(j)}\\
		\lesssim
		\bigg\lVert
		\int_{-\frac{1}{2}}^{\frac{1}{2}} e^{-2\pi i \xi j} \big(1 - e^{2\pi i \xi (l-l')}\big)
		\eta_s(\xi) \hat{f}(\xi) d\xi
		\bigg\rVert_{\ell^r(j)}.
	\end{multline*}
	Since $\eta_s = \eta_s \eta_{s-1}$, by Minkowski's inequality and Lemma \ref{lem:1} we obtain
	that the last expression can be dominated by
	$$
	\bigg\lVert
	\int_{-\frac{1}{2}}^{\frac{1}{2}}
	e^{-2\pi i \xi j} \big(1 - e^{2\pi i \xi (l-l')}\big) \eta_{s-1}(\xi) d\xi
	\bigg\rVert_{\ell^1(j)}
	\big\lVert
	\mathcal{F}^{-1} \big(\eta_s \hat{f}\big)
	\big\rVert_{\ell^r}
	\leq
	C q A^{-s}
	\big\lVert
	\mathcal{F}^{-1} \big(\eta_s \hat{f}\big)
	\big\rVert_{\ell^r}.
	$$
	Therefore, by \eqref{eq:18}
	$$
	J_l \leq J_{l'} +
	q^{-1} \big\lVert \mathcal{F}^{-1} \big(\eta_s \hat{f}\big) \big\rVert_{\ell^r}.
	$$
	Summing up over all $l' \in \{1, 2, \ldots, q\}$ we obtain
	$$
	q J_l^r \leq 2^{r-1} I^r + C 2^{r-1} q^{1-r}
	\big\lVert \mathcal{F}^{-1} \big(\eta_s \hat{f}\big) \big\rVert_{\ell^r}^r
	\lesssim
	\big\lVert \mathcal{F}^{-1} \big(\eta_s \hat{f}\big) \big\rVert_{\ell^r}^r.
	$$
	Eventually, by Lemma \ref{lem:2} we conclude
	\begin{equation}
		\label{eq:16}
		\Big\lVert
		\sup_{k \in \NN} \big\lvert \mathcal{F}^{-1} \big(\Psi_k \eta_s \hat{f} \big)(q j + l)
		\big\rvert
		\Big\rVert_{\ell^r(j)}
		\lesssim
		\big\lVert \mathcal{F}^{-1} \big(\eta_s \hat{f}\big)(qj+l) \big\rVert_{\ell^r(j)}.
	\end{equation}
	Next, we resume the analysis of \eqref{eqm:14}. Using \eqref{eq:16} we get
	$$
	\bigg(
	\sum_{l=1}^q
	\Big\lVert
	\sup_{k \in \NN}
	\big\lvert
	\mathcal{F}^{-1} \big(\Psi_k \eta_s F_b(\cdot\ ; l) \big) (q j + l)
	\big\rvert
	\Big\rVert_{\ell^r(j)}^r
	\bigg)^{1/r}
	\lesssim
	\bigg(
	\sum_{l=1}^q
	\Big\lVert
	\mathcal{F}^{-1} \big(\eta_s F_b(\cdot\ ; l) \big) (q j + l)
	\Big\rVert_{\ell^r(j)}^r
	\bigg)^{1/r}.
	$$
	We observe that by the change of variables
	$$
	\mathcal{F}^{-1}\big(\eta_s F_b(\cdot\ ; l)\big)(qj+l)
	=\sum_{a=1}^b \mathcal{F}^{-1}\big(\eta_s(\cdot - a/b) \hat{f}\big)(qj+l).
	$$
	Thus by Minkowski's inequality
	$$
	\bigg(
	\sum_{l=1}^q
	\Big\lVert
	\mathcal{F}^{-1} \big(\eta_s F_b(\cdot\ ; l) \big) (q j + l)
	\Big\rVert_{\ell^r(j)}^r
	\bigg)^{1/r}
	\leq
	\Big\lVert
	\mathcal{F}^{-1} \Big(\sum_{a=1}^b \eta_s(\cdot - a/b) \Big)
	\Big\rVert_{\ell^1}
	\norm{f}_{\ell^r}.
	$$
	Since for $j \in \ZZ$
	$$
	\sum_{a=1}^b e^{-2\pi i j a/b} =
	\begin{cases}
		b & \text{if } b \mid j,\\
		0 & \text{otherwise}
	\end{cases}
	$$
	we conclude
	$$
	\Big\lVert
	\mathcal{F}^{-1} \Big(\sum_{a=1}^b \eta_s(\cdot - a/b) \Big)
	\Big\rVert_{\ell^1}
	=
	\Big\lVert
	\mathcal{F}^{-1} \eta_s(j)
	\sum_{a=1}^b e^{-2\pi i j a/b}
	\Big\rVert_{\ell^1(j)}
	=b \norm{\mathcal{F}^{-1} \eta_s(bj)}_{\ell^1(j)}.
	$$
	Now Lemma \ref{lem:1} and Lemma \ref{lem:2}
	imply
	$$
	b \norm{\mathcal{F}^{-1} \eta_s(bj)}_{\ell^1(j)}
	\lesssim \norm{\mathcal{F}^{-1} \eta_s}_{\ell^1}
	\lesssim 1.
	$$
	This completes the proof of \eqref{eqm:14}. Eventually, by \eqref{eq:5} and \eqref{eqm:12} we
	obtain that
	\begin{equation}
		\label{eq:4}
		\Big\lVert
		\sup_{k \in \NN}
		\Big\lvert
		\sum_{j=0}^k
		\mathcal{F}^{-1} (\nu_j^s \hat{f})
		\Big\rvert
		\Big\rVert_{\ell^r}
		\lesssim
		2^{\epsilon s} \norm{f}_{\ell^r}
	\end{equation}
	for any $\epsilon > 0$ and $s \in \NN$. If $r = 2$ we may refine the estimate \eqref{eq:4}
	(see also \cite{Bou1}). Let
    $$
	G_q(\xi) = \sum_{a \in A_q} \eta_{s-1}(\xi - a/q) \hat{f}(\xi).
	$$
    and note that
    \begin{align*}
    	\sum_{a \in A_q} \mathcal{F}^{-1} \big(\Psi_k(\cdot -a/q) \eta_s(\cdot-a/q) \hat{f} \big)
		=\sum_{a \in A_q} \mathcal{F}^{-1} \big(\Psi_k(\cdot - a/q) \eta_s(\cdot -a/q) G_q\big)
    \end{align*}
	since $\eta_s = \eta_s \eta_{s-1}$, and the supports of $\eta_s(\cdot -a/q)$'s are disjoint
	when $a/q$ varies. By \eqref{eqm:12} we have
	$$
	\Big\lVert
	\sup_{k \in \NN}
	\Big\lvert
	\sum_{a \in A_q}
	\mathcal{F}^{-1} (\Psi_k(\cdot -a/q) \eta_s(\cdot -a/q) G_q)
	\Big\rvert
	\Big\rVert_{\ell^2}
	\lesssim
	q^{\epsilon} \norm{\mathcal{F}^{-1} G_q}_{\ell^2}
	$$
	whereas by \eqref{eq:5}, we have
	$$
	\Big\lVert
	\sup_{k \in \NN}
	\Big\lvert
	\sum_{j=0}^k
	\mathcal{F}^{-1}\big(\nu^s_j \hat{f}\big)
	\Big\rvert
	\Big\rVert_{\ell^2}
	\leq
	\sum_{q = 2^s}^{2^{s+1}-1}q^{-1+\epsilon}\Big\lVert
	\sup_{k \in \NN}
	\Big\lvert
	\sum_{a \in A_q}
	\mathcal{F}^{-1} (\Psi_k(\cdot -a/q) \eta_s(\cdot -a/q) \hat{f})
	\Big\rvert
	\Big\rVert_{\ell^2}.
	$$
	These two bounds yield
	$$
	\Big\lVert
	\sup_{k \in \NN}
	\Big\lvert
	\sum_{j=0}^k
	\mathcal{F}^{-1}\big(\nu^s_j \hat{f}\big)
	\Big\rvert
	\Big\rVert_{\ell^2}
	\lesssim
	\sum_{q = 2^s}^{2^{s+1}-1} q^{-1+2\epsilon} \norm{\mathcal{F}^{-1} G_q}_{\ell^2}\\
	\lesssim
	2^{-s/2 + 2\epsilon s}
	\Big( \sum_{a/q \in \mathscr{R}_s}
	\big\lVert
	\mathcal{F}^{-1} \big(\eta_{s-1}(\cdot-a/q) \hat{f}\big)
	\big\rVert_{\ell^2}^2
	\Big)^{1/2},
	$$
    where the last estimate follows from Cauchy--Schwartz inequality and the definition of $G_q$.
	Eventually, by Plancherel's theorem we may write
	$$
	\sum_{a/q \in \mathscr{R}_s}
	\big\lVert \mathcal{F}^{-1}
	\big(\eta_{s-1}(\cdot-a/q) \hat{f}\big)
	\big\rVert_{\ell^2}^2
	=
	\sum_{a/q \in \mathscr{R}_s}
	\int_{\RR}
	\abs{\eta_{s-1}(\xi - a/q)}^2
	\big\lvert \hat{f}(\xi) \big\rvert^2
	d\xi
	$$
	which is majorized by $\norm{f}_{\ell^2}^2$. Thus for appropriately chosen $\epsilon>0$ we
	obtain
	\begin{equation}
		\label{eq:23}
		\Big\lVert
		\sup_{k \in \NN}
		\Big\lvert
		\sum_{j=0}^k
		\mathcal{F}^{-1}\big(\nu_j^s \hat{f}\big)
		\Big\rvert
		\Big\rVert_{\ell^2} \leq 2^{-s/4} \norm{f}_{\ell^2}.
	\end{equation}
	Next, for $r \neq 2$ we can use Marcinkiewicz interpolation theorem and interpolate between
	\eqref{eq:4} and \eqref{eq:23} to conclude the proof.
\end{proof}

\subsection{Maximal function}
\label{subsec:5}
We have gathered necessary tools to illustrate the proof of Theorem \ref{thm:2}. First, we
show the boundedness on $\ell^r(\ZZ)$ of the maximal function $T^*$.
\begin{theorem}
	\label{th:4}
	The maximal function $T^*$ is bounded on $\ell^r(\ZZ)$ for each $1 < r < \infty$.
\end{theorem}
\begin{proof}
	Let us observe that for a non-negative function $f$
	$$
	T^* f(n) \lesssim
	\sup_{k \in \NN}
	\Big\lvert
	\sum_{j = 0}^k \mathcal{F}^{-1}\big(m_j \hat{f}\big)(n)
	\Big\rvert
	+
	\mathcal{M} f(n)
	$$
	where $\mathcal{M} f = \sup_{N \in \NN} \abs{A_N f}$ is a maximal function corresponding with
	Bourgain--Wierdl's averages
	$$
	A_N f(n) = N^{-1} \sum_{p \in \pm \PP_N} f(n-p) \log \abs{p}.
	$$
	Indeed, suppose $\tau^k \leq N < \tau^{k+1}$ for $k \in \NN$. Then
	$$
	T_N f(n) =
	\sum_{j=0}^k \sum_{p \in \pm \PP} f(n-p) K_j(p) \log \abs{p}
	-\sum_{p \in \pm R_N} f(n-p) K(p) \log \abs{p}.
	$$
	where $R_N = \PP \cap (N, \tau^{k+1})$. Therefore, by \eqref{eq:46}, we see
	$$
	\Big\lvert
	\sum_{p \in R_N} f(n-p) K(p) \log \abs{p}
	\Big\rvert
	\lesssim
	\tau^{-k}
	\sum_{p \in \pm \PP_{\tau^{k+1}}} f(n-p) \log |p| \lesssim A_{\tau^{k+1}} f(n).
	$$
	Since the maximal function $\mathcal{M}$ is bounded on $\ell^r(\ZZ)$ for any $r > 1$
	(see \cite{Bou} or Appendix \ref{apx:1}) thus we have reduced the boundedness of $T^*$ to
	proving
	$$
	\Big\lVert
	\sup_{k \in \NN}
	\Big\lvert
	\sum_{j=0}^k
	\mathcal{F}^{-1}
	\big(m_j \hat{f} \big)
	\Big\rvert
	\Big\rVert_{\ell^r} \lesssim \norm{f}_{\ell^r}.
	$$
	Let us consider $f \in \ell^r(\ZZ)$ for $r>1$. By Theorem \ref{th:3} we know that for
	$j\in\NN$
	\begin{multline*}
		\big\lVert
		\mathcal{F}^{-1} \big(\nu_j \hat{f} \big)
		\big\rVert_{\ell^r}
		\leq
		\sum_{s \in \NN}
		\big\lVert
		\mathcal{F}^{-1} \big(\nu_j^s \hat{f}\big)
		\big\rVert_{\ell^r}
		\leq
		\sum_{s \in \NN}\Big\|\sup_{k\in\NN}\Big|\sum_{j=0}^k
		\mathcal{F}^{-1} \big(\nu_j^s \hat{f}\big)-\sum_{j=0}^{k-1}
		\mathcal{F}^{-1} \big(\nu_j^s \hat{f}\big)\Big|
		\Big\|_{\ell^r}\\
		\lesssim\sum_{s \in \NN}\Big\|\sup_{k\in\NN}\Big|\sum_{j=0}^k
		\mathcal{F}^{-1} \big(\nu_j^s \hat{f}\big)\Big|
		\Big\|_{\ell^r}
		\lesssim \sum_{s \in \NN} 2^{-\delta s} \norm{f}_{\ell^r}
		\lesssim \norm{f}_{\ell^r}.
	\end{multline*}
	If $f$ is non-negative then
	$$
	\Big\lvert
	\sum_{p \in \pm \PP} f(x-p) K_j(p) \log \abs{p}
	\Big\rvert
	\leq
	\tau^{-j+1}
	\sum_{p \in \pm \PP_{\tau^{j+1}}} f(x-p) \log \abs{p}
	$$
	thus
	$$
	\big\lVert \mathcal{F}^{-1} \big(m_j \hat{f}\big) \big\rVert_{\ell^r}
	\lesssim \tau^{-j} \Big(\sum_{p \in \PP_{\tau^{j+1}}} \log p \Big) \norm{f}_{\ell^r}
	\lesssim \norm{f}_{\ell^r}.
	$$
	Hence,
	\begin{equation}
		\label{eq:24}
		\Big\lVert
		\mathcal{F}^{-1} \big((m_j - \nu_j)\hat{f} \big)
		\Big\rVert_{\ell^r}
		\lesssim \norm{f}_{\ell^r}.
	\end{equation}
	For $r = 2$ we use Proposition \ref{prop:2} to get
	\begin{equation}
		\label{eq:25}
		\Big\lVert
		\mathcal{F}^{-1}\big((m_j - \nu_j) \hat{f} \big)
		\Big\rVert_{\ell^2}
		\leq
		\norm{m_j - \nu_j}_{L^\infty} \norm{f}_{\ell^2}
		\lesssim j^{-\alpha} \norm{f}_{\ell^2}
	\end{equation}
	for any $\alpha>0$ big enough. If $r \neq 2$ we apply Marcinkiewicz interpolation theorem
	to interpolate between \eqref{eq:24} and \eqref{eq:25} and obtain
	\begin{equation}
		\label{eq:26}
		\Big\lVert
		\mathcal{F}^{-1} \big((m_j - \nu_j)\hat{f} \big)
		\Big\rVert_{\ell^r}
		\lesssim j^{-2} \norm{f}_{\ell^r}.
	\end{equation}
	Since
	$$
	\Big\lVert
	\sup_{k \in \NN}
	\Big\lvert
	\sum_{j=0}^k
	\mathcal{F}^{-1} \big((m_j - \nu_j)\hat{f} \big)
	\Big\rvert
	\Big\rVert_{\ell^r}
	\leq
	\sum_{j \in \NN}
	\Big\lVert
	\mathcal{F}^{-1} \big((m_j - \nu_j) \hat{f} \big)
	\Big\rVert_{\ell^r}
	$$
	by \eqref{eq:26} and Theorem \ref{th:3} we finish the proof.
\end{proof}
Next, we demonstrate the pointwise convergence of $\seq{T_N}{N \in \NN}$.
\begin{proposition}
	If $f \in \ell^r(\ZZ)$, $1 < r < \infty$ then for every $n \in \ZZ$
	\begin{align}
		\label{eqm:15}
		 \lim_{N \rightarrow \infty} T_N f(n) = T f (n)
	\end{align}
	and $T$ is bounded on $\ell^r(\ZZ)$.
\end{proposition}
\begin{proof}
	If $N \in \NN$ we define an operator $T^N$ by setting
	$$
	T^N f(n) =
	\sum_{\atop{p\in\pm\PP}{|p|> N}}
	f(x-p) K(p) \log|p|
	$$
	for any $f \in \ell^r(\ZZ)$. By H\"{o}lder's inequality we see that for every
	$n \in \ZZ$
	$$
	\abs{T^Nf(n)}
	\leq
	2
	\Big( \sum_{\atop{p \in \PP}{p > N}}
	\big( p^{-1} \log p \big)^{r'} \Big)^{1/{r'}}
	\norm{f}_{\ell^r}
	$$
	where $r'$ stands for the conjugate exponent to $r$, i.e. $1/r+1/r'=1$. The last inequality
	shows that, on the one hand, $T$ is well defined for any $f\in\ell^r(\ZZ)$, on the other --
	proves \eqref{eqm:15}. Next, Fatou's lemma with boundedness of $T^*$ yield
	$$
	\norm{Tf}_{\ell^r}
	= \norm{\liminf_{N\to\infty}T_N f}_{\ell^r}
	\leq \liminf_{N\to\infty} \lVert T_Nf \rVert_{\ell^r}
	\leq \lVert T^*f \rVert_{\ell^r}
	\lesssim_r \norm{f}_{\ell^r}
	$$
	which completes the proof.
\end{proof}

\subsection{Oscillatory norm for $H_N$}
Let $\seq{N_j}{j \in \NN}$ be a strictly increasing sequence of $\Lambda$ elements. We set
$N_j = \tau^{k_j}$ and $\Lambda_j = \Lambda \cap (N_j, N_{j+1}]$. In this Section we consider
the kernel $K(x) = x^{-1}$. Since each $K_j$ for $j \in \NN$ has mean zero we have
\begin{equation}
	\label{eq:40}
	\abs{\Phi_j(\xi)}
	\leq
	\int_{\RR} \abs{1 - e^{2\pi i \xi x}} \abs{K_j(x)} dx \lesssim \abs{\xi} \tau^j.
\end{equation}
Let $H_N$ denote the truncated Hilbert transform
$$
H_N f(n) = \sum_{p \in \pm \PP_N} \frac{f(n - p)}{p} \log \abs{p}.
$$
The following argument is based on \cite[Section 7]{Bou1}.
\begin{proposition}
	\label{prop:5}
	There is $C > 0$ such that for every $J \in \NN$ and $s\in\NN$ we have
	$$
	\sum_{j=0}^J
	\big\lVert
	\sup_{\tau^k \in \Lambda_j}
	\big\lvert
	\mathcal{F}^{-1}\big((\Psi_k -\Psi_{k_j}) \eta_s \hat{f} \big)
	\big\rvert
	\big\rVert_{\ell^2}^2
	\leq C \big\|\mathcal{F}^{-1}\big(\eta_s \hat{f}\big)\big\|_{\ell^2}^2.
	$$
\end{proposition}
\begin{proof}
	Let $B_j = \{x \in (-1/2, 1/2): \abs{x} \leq N_j^{-1}\}$. By Plancherel's theorem we have
	\begin{multline*}
		\sum_{j=0}^J
		\big\lVert
		\sup_{\tau^k \in \Lambda_j}
		\big\lvert
		\mathcal{F}^{-1} \big((\Psi_k - \Psi_{k_j}) \ind{B_{j+1}} \eta_s \hat{f}\big)
		\big\rvert
		\big\rVert_{\ell^2}^2
		\leq
		\sum_{j=0}^J
		\sum_{k = k_j}^{k_{j+1}}
		\big\lVert
		\mathcal{F}^{-1}\big((\Psi_k - \Psi_{k_j}) \ind{B_{j+1}} \eta_s \hat{f}\big)
		\big\rVert_{\ell^2}^2\\
		\leq
		\Big\lVert
		\sum_{j=0}^J \ind{B_{j+1}}
		\sum_{k=k_j}^{k_{j+1}} \abs{\Psi_k - \Psi_{k_j}}^2
		\Big\rVert_{L^\infty}
		\big\|\mathcal{F}^{-1}\big(\eta_s \hat{f}\big)\big\|_{\ell^2}^2.
	\end{multline*}
	By \eqref{eq:40} we have
	$$
	\abs{\Psi_k(\xi) - \Psi_{k_j}(\xi)}
	= \Big\lvert \sum_{l = k_j+1}^k \Phi_l(\xi) \Big\rvert
	\lesssim \abs{\xi} \tau^k.
	$$
	Hence,
	$$
	\sum_{j=0}^J \ind{B_{j+1}}(\xi)
	\sum_{k = k_j}^{k_{j+1}}
	\abs{\Psi_k(\xi) - \Psi_{k_j}(\xi)}^2
	\lesssim
	\abs{\xi}^2
	\sum_{j=0}^J
	\ind{B_{j+1}}(\xi)
	\sum_{k = k_j}^{k_{j+1}} \tau^{2k}
	\lesssim
	\abs{\xi}^2
	\sum_{j: N_{j+1} \leq \abs{\xi}^{-1}} N_{j+1}^2
	\lesssim 1.
	$$
	Therefore, we obtain
	$$
	\sum_{j=0}^J
	\big\lVert
	\sup_{\tau^k \in \Lambda_j}
	\big\lvert
	\mathcal{F}^{-1} \big((\Psi_k - \Psi_{k_j}) \ind{B_{j+1}} \eta_s \hat{f}\big)
	\big\rvert
	\big\rVert_{\ell^2}^2
	\lesssim
	\big\|\mathcal{F}^{-1}\big(\eta_s \hat{f}\big)\big\|_{\ell^2}^2.
	$$
	Similar for $B_j^c$, replacing $\Psi_{k_j}$ by $\Psi_{k_{j+1}}$ under the supremum, we can estimate
	\begin{multline*}
		\sum_{j=0}^J
		\big\lVert
		\sup_{\tau^k \in \Lambda_j}
		\big\lvert
		\mathcal{F}^{-1} \big((\Psi_k - \Psi_{k_j}) \ind{B_j^c} \eta_s \hat{f}\big)
		\big\rvert
		\big\rVert_{\ell^2}^2
		\leq
		\sum_{j=0}^J
		\sum_{k = k_j}^{k_{j+1}}
		\big\lVert
		\mathcal{F}^{-1}\big((\Psi_{k_{j+1}} - \Psi_k) \ind{B_j^c} \eta_s \hat{f}\big)
		\big\rVert_{\ell^2}^2\\
		\leq
		\Big\lVert
		\sum_{j=0}^J \ind{B_j^c}
		\sum_{k=k_j}^{k_{j+1}} \abs{\Psi_{k_{j+1}} - \Psi_k}^2
		\Big\rVert_{L^\infty}
		\big\|\mathcal{F}^{-1}\big(\eta_s \hat{f}\big)\big\|_{\ell^2}^2.
	\end{multline*}
	Now, using \eqref{eq:39} we get
	$$
	\abs{\Psi_{k_{j+1}}(\xi) - \Psi_k(\xi)}
	\lesssim \abs{\xi}^{-1} \tau^{-k}
	$$
	thus
	$$
	\sum_{j=0}^J \ind{B_j^c}(\xi)
	\sum_{k = k_j}^{k_{j+1}}
	\abs{\Psi_{k_{j+1}}(\xi) - \Psi_k(\xi)}^2
	\lesssim
	\abs{\xi}^{-2}
	\sum_{j=0}^J
	\ind{B_j^c}(\xi)
	\sum_{k = k_j}^{k_{j+1}} \tau^{-2k}\\
	\lesssim
	\abs{\xi}^{-2}
	\sum_{j: N_j \geq \abs{\xi}^{-1}} N_j^{-2}
	\lesssim 1.
	$$
	Therefore, we conclude
	$$
	\sum_{j=0}^J
	\big\lVert
	\sup_{\tau^k \in \Lambda_j}
	\big\lvert
	\mathcal{F}^{-1} \big((\Psi_k - \Psi_{k_j}) \ind{B_j^c} \eta_s \hat{f}\big)
	\big\rvert
	\big\rVert_{\ell^2}^2
	\lesssim
	\big\|\mathcal{F}^{-1}\big(\eta_s \hat{f}\big)\big\|_{\ell^2}^2.
	$$
	Eventually, by Proposition \ref{prop:3}
	$$
	\sum_{j=0}^J
	\Big\lVert
	\sup_{\tau^k \in \Lambda_j}
	\big\lvert
	\mathcal{F}^{-1}\big((\Psi_k - \Psi_{k_j}) \ind{B_j} \ind{B_{j+1}^c} \eta_s \hat{f} \big)
	\big\rvert
	\Big\rVert_{\ell^2}^2
	\lesssim
	\sum_{j=0}^J
	\norm{\mathcal{F}^{-1} \big(\ind{B_j} \ind{B_{j+1}^c} \eta_s\hat{f}\big)}_{\ell^2}^2
	$$
	which is bounded by $\big\|\mathcal{F}^{-1}\big(\eta_s \hat{f}\big)\big\|_{\ell^2}^2$.
\end{proof}

\begin{theorem}
	\label{th:5}
	For every $J \in \NN$ there is $C_J$ such that
	$$
	\sum_{j=0}^J
	\Big\lVert
	\sup_{\tau^k \in \Lambda_j}
	\big\lvert
	H_{\tau^k } f - H_{N_j} f
	\big\rvert
	\Big\rVert_{\ell^2}^2
	\leq
	C_J \norm{f}_{\ell^2}^2
	$$
	and $\lim_{J \to \infty} C_J/J = 0$.
\end{theorem}
\begin{proof}
	By Proposition \ref{prop:2}, we have
	$$
	\sum_{j=0}^J
	\Big\rVert
	\sup_{\tau^k \in \Lambda_j}
	\Big\lvert
	\sum_{l = k_j+1}^k
	\mathcal{F}^{-1} \big((m_l - \nu_l)\hat{f}\big)
	\Big\rvert
	\Big\lVert_{\ell^2}^2
	\lesssim
	\Big(
	\sum_{j=0}^J
	\sum_{l = k_j+1}^{k_{j+1}} l^{-2}
	\Big)^{1/2}
	\norm{f}_{\ell^2}^2\lesssim \norm{f}_{\ell^2}^2.
	$$
	Consequently, it is enough to demonstrate
	$$
	\sum_{j = 0}^J
	\Big\lVert
	\sup_{\tau^k \in \Lambda_j}
	\Big\lvert
	\sum_{l = k_j+1}^k \mathcal{F}^{-1}\big(\nu_l \hat{f}\big)
	\Big\rvert
	\Big\rVert_{\ell^2}^2
	\leq C_J \norm{f}_{\ell^2}^2
	$$
	where $\lim_{J \to \infty} C_J/J = 0$.
	
	Let $s_0 \in \NN$ be defined as $2^{s_0} \leq J^{1/3} < 2^{s_0+1}$. By Theorem \ref{th:3}
	we have
	$$
	\Big\lVert
	\sup_{\tau^k \in \Lambda_j}
	\Big\lvert
	\sum_{s = s_0}^{\infty} \sum_{l=k_j+1}^k \mathcal{F}^{-1}\big(\nu_l^s \hat{f} \big)
	\Big\rvert
	\Big\rVert_{\ell^2}
	\lesssim
	\sum_{s = s_0}^\infty
	\Big\lVert
	\sup_{k \in \NN}
	\Big\lvert
	\sum_{l=0}^k \mathcal{F}^{-1}\big(\nu_l^s \hat{f} \big)
	\Big\rvert
	\Big\rVert_{\ell^2}
	\lesssim J^{-\delta/3} \norm{f}_{\ell^2}.
	$$
	We set
	$$
	D_J = \sum_{s=0}^{s_0-1} \sum_{a/q \in \mathscr{R}_s} \frac{1}{\varphi(q)}.
	$$
	By the change of variables, Cauchy--Schwartz inequality and  by Proposition \ref{prop:5} we get
    \begin{multline*}
    	\sum_{j=0}^J
		\Big\lVert
		\sup_{\tau^k \in \Lambda_j}
		\Big\lvert
		\sum_{s = 0}^{s_0-1} \sum_{l=k_j+1}^k \mathcal{F}^{-1}\big(\nu_l^s \hat{f} \big)
		\Big\rvert
		\Big\rVert_{\ell^2}^2\\
	    \le\sum_{j = 0}^J\bigg(\sum_{s = 0}^{s_0-1}\sum_{a/q \in \mathscr{R}_s}\frac{1}{\varphi(q)}
	    \Big\lVert
		\sup_{\tau^k \in \Lambda_j}
		\Big\lvert
		\sum_{l=k_j+1}^k \mathcal{F}^{-1}\big(\Phi_l \eta_s \hat{f}(\cdot + a/q)\big)
		\Big\rvert\Big\rVert_{\ell^2}\bigg)^2\\
		\leq
		D_J^2
		\sum_{s = 0}^{s_0-1} \sum_{a/q \in \mathscr{R}_s}
		\sum_{j = 0}^J
		\Big\lVert
		\sup_{\tau^k \in \Lambda_j}
		\Big\lvert
		 \mathcal{F}^{-1}\big((\Psi_k-\Psi_{k_j}) \eta_s \hat{f}(\cdot + a/q)\big)
		\Big\rvert
		\Big\rVert_{\ell^2}^2\\
		\lesssim D_J^2\sum_{s = 0}^{s_0-1} \sum_{a/q \in \mathscr{R}_s}
		\big\|
		\mathcal{F}^{-1}\big(\eta_s(\cdot - a/q) \hat{f}\big)
		\big\|_{\ell^2}^2\lesssim D_J^2 s_0 \|f\|_{\ell^2}^2.
    \end{multline*}
	By the definition of $\mathscr{R}_s$ we see that $D_J\lesssim2^{s_0}\le J^{1/3}$
	thus we achieve
	$$
	\sum_{j = 0}^J
	\Big\lVert
	\sup_{\tau^k \in \Lambda_j}
	\Big\lvert
	\sum_{l = k_j+1}^k \mathcal{F}^{-1}\big(\nu_l \hat{f}\big)
	\Big\rvert
	\Big\rVert_{\ell^2}^2
	\lesssim
	J \big(J^{-\delta/3} + J^{-1/3} \log J \big) \norm{f}_{\ell^2}^2
	$$
	which finishes the proof.
\end{proof}

\section{Dynamical systems}
Let $(X, \mathcal{B}, \mu, S)$ be a dynamical system on a measure space $X$. Let
$S: X \rightarrow X$ be an invertible measure preserving transformation. For $N > 0$ we
set
$$
\mathcal{H}_N f (x) = \sum_{p \in \pm \PP_N} \frac{f(S^{-p} x)}{p}\log \abs{p}.
$$
We are going to show Theorem \ref{thm:1}. We start from oscillatory norm.
\begin{proposition}
	\label{prop:4}
	For each $J \in \NN$ there is $C_J$ such that
	$$
	\sum_{j = 0}^J
	\big\lVert
	\sup_{N \in \Lambda_j}
	\big\lvert
	\mathcal{H}_N f - \mathcal{H}_{N_j} f
	\big\rvert
	\big\rVert_{L^2(\mu)}^2 \leq C_J \norm{f}_{L^2(\mu)}^2
	$$
	and $\lim_{J \to \infty} C_J/J = 0$.
\end{proposition}
\begin{proof}
	Let $R \geq N_{J}$. For a fixed $x \in X$ we define a function on $\ZZ$ by
	$$
	\phi(n) =
	\begin{cases}
		f(S^n x) & \abs{n} \leq R, \\
		0 & \text{otherwise.}
	\end{cases}
	$$
	Then for $\abs{n} \leq  R - N$
	$$
	\mathcal{H}_N f(S^n x) = \sum_{p \in \pm \PP_N} \frac{f(S^{n-p}x)}{p} \log \abs{p}
	= \sum_{p \in \pm \PP_N} \frac{\phi(n-p)}{p} \log \abs{p} = H_N \phi(n).
	$$
	Hence,
	$$
	\sum_{\abs{n} = 0}^{R - N_J}
	\sup_{N \in \Lambda_j}
	\big\lvert
	\mathcal{H}_N f(S^n x) - \mathcal{H}_{N_j} f(S^n x)
	\big\rvert^2
	\leq
	\big\lVert
	\sup_{N \in \Lambda_j}
	\big\lvert
	H_N \phi - H_{N_j} \phi
	\big\rvert
	\big\rVert_{\ell^2}^2.
	$$
	Therefore, by Theorem \ref{th:5} we can estimate
	$$
	\sum_{\abs{n} = 0}^{R - N_J} \sum_{j = 0}^J
	\sup_{N \in \Lambda_j}
	\big\lvert
	\mathcal{H}_N f(S^n x) - \mathcal{H}_{N_j} f(S^n x)
	\big\rvert^2
	\leq
	C_J \norm{\phi}_{\ell^2}^2
	=
	C_J \sum_{\abs{n}=0}^R \abs{f(S^n x)}^2.
	$$
	Since $S$ is a measure preserving transformation integration with respect to $x \in X$
	implies
	$$
	(R - N_J) \sum_{j = 0}^J
	\big\lVert
	\sup_{N \in \Lambda_j}
	\big\lvert
	\mathcal{H}_N f - \mathcal{H}_{N_j} f
	\big\rvert
	\big\rVert_{L^2(\mu)}^2
	\leq
	C_J R \norm{f}_{L^2(\mu)}^2.
	$$
	Eventually, if we divide both sides by $R$ and take $R \rightarrow \infty$ we conclude
	the proof.
\end{proof}

\begin{corollary}
	\label{cor:1}
	The maximal function
	$$
	\mathcal{H}^* f(x) = \sup_{N \in \NN} \big\lvert \mathcal{H}_N f (x)\big\rvert
	$$
	is bounded on $L^r(\mu)$ for each $1 < r < \infty$.
\end{corollary}

Next, we show the pointwise convergence of $\seq{\mathcal{H}_N}{N \in \NN}$.
\begin{theorem}
	Let $f \in L^r(\mu)$, $1 < r < \infty$. For $\mu$-almost every $x \in X$
	$$
	\lim_{N \to \infty} \mathcal{H}_N f(x) = \mathcal{H} f(x)
	$$
	and $\mathcal{H}$ is bounded on $L^r(\mu)$.
\end{theorem}
\begin{proof}
	Let $f \in L^2(\mu)$, since the maximal function $\mathcal{H}^*$ is bounded on $L^2(\mu)$
	we may assume $f$ is bounded by $1$. Suppose $\seq{\mathcal{H}_N f}{N \in \NN}$ does not
	converge $\mu$-almost everywhere. Then there is $\epsilon > 0$ such that
	$$
	\mu\big\{x\in X : \limsup_{M, N\to\infty}
		\big\lvert
		\mathcal{H}_N f(x) - \mathcal{H}_{M} f(x)
		\big\rvert
		> 4\epsilon
	\big\} > 4\epsilon.
	$$
	Now one can find a strictly increasing sequence of integers $\seq{k_j}{j\in\NN}$ such that
	for each $j \in \NN$
	$$
	\mu\big\{x\in X : \sup_{N_j \leq N \leq N_{j+1}}
		\big\lvert
		\mathcal{H}_N f(x) - \mathcal{H}_{N_j} f(x)
		\big\rvert
		> \epsilon
	\big\} > \epsilon
	$$
	where $N_j = \tau^{k_j}$ and $\tau=1+\epsilon/4$. If $\tau^k \leq N < \tau^{k+1}$ then setting
	$P_k = \PP \cap (\tau^k, \tau^{k+1}]$ we get
	$$
	\big\lvert
	\mathcal{H}_N f(x) - \mathcal{H}_{\tau^k} f(x)
	\big\rvert
	\leq
	\tau^{-k} \sum_{p \in P_k} \log p.
	$$
	By Siegel--Walfisz theorem we get
    $$
	\sum_{p \in \PP_N} \log p
	=N + \mathcal{O}(N (\log N)^{-1})
	$$
	thus there is $C > 0$ such that
	$$
	\Big\lvert
	\tau^{-k}
	\sum_{p \in P_k} \log p
	-\tau + 1
	\Big\rvert
	\leq
	C k^{-1} (\log \tau)^{-1}.
	$$
	Hence, whenever $k \geq 4 C \epsilon^{-1} (\log \tau)^{-1}$ we have
	$$
	\big\lvert \mathcal{H}_N f(x) - \mathcal{H}_{\tau^k} f (x) \big\rvert \leq \epsilon/2.
	$$
	In particular, we conclude
	$$
	\mu\big\{x\in X: \sup_{\tau^k  \in \Lambda_j}
		\big\lvert \mathcal{H}_{\tau^k } f(x) - \mathcal{H}_{N_j} f(x)\big\rvert > \epsilon/2
	\big\} > \epsilon
	$$
	for each $k_j \geq 4C \epsilon^{-1} (\log \tau)^{-1}$ which contradicts to Proposition
	\ref{prop:4}. Indeed, 
	\begin{align*}
		\epsilon^3
		\lesssim
		\frac{1}{J-J_0}\sum_{j = 0}^J
		\big\lVert
		\sup_{\tau^k  \in \Lambda_j}
		\big\lvert
		\mathcal{H}_{\tau^k } f - \mathcal{H}_{N_j} f
		\big\rvert
		\big\rVert_{L^2(\mu)}^2
		\leq
		\frac{C_J}{J-J_0} \norm{f}_{L^2(\mu)}^2
		\end{align*}
	where $J_0=\min\{j\in\NN: k_j\ge 4C \epsilon^{-1} (\log \tau)^{-1}\}$. Now, the standard
	density argument implies pointwise convergence for each $f\in L^r(\mu)$ where $r>1$, and
	the proof of the theorem is completed.
\end{proof}

\appendix
\section{Boundedness of $\mathcal{M}$}
\label{apx:1}
In the Appendix we discuss why the maximal function
$$
\mathcal{M} f(n) = \sup_{N \in \NN}
\Big\lvert
N^{-1}
\sum_{p \in \pm \PP_N} f(n - p) \log \abs{p}
\Big\rvert
$$
is bounded on $\ell^r(\ZZ)$. This fact was published by Wierdl in \cite{wrl}, however, on page
331 in the last equality for ** the factor $q$ has the power $1$ in place of $p$. Therefore, it is
not sufficient to show an estimate (24) from \cite{wrl} to conclude the proof. In fact, one has to
prove the estimate corresponding to \eqref{eq:16} from the present paper.

For the completeness we provide the sketch of the proof based on the method used in
Section \ref{sec:3}. First, we may restrict supremum to dyadic $N$. We modify the definition of
the multiplier $m_j$ by setting
$$
m_j(\xi) = 2^{-j} \sum_{p \in \pm \PP_N} e^{2\pi i \xi p} \log \abs{p}.
$$
Hence, it suffices to show that for $r > 1$
$$
\big\lVert
\sup_{k \in \NN}
\big\lvert
\mathcal{F}^{-1} \big(m_k \hat{f}\big)
\big\rvert
\big\rVert_{\ell^r}
\lesssim
\norm{f}_{\ell^r}.
$$
Keeping the definition of the major arcs and setting
$$
\Psi_j(\xi) = 2^{-j} \int\limits_{1 \leq \abs{x} \leq 2^j} e^{2\pi i \xi x} dx
$$
Proposition \ref{prop:1} holds true. For proof we use the well-known result that for
$\xi \in \mathfrak{M}_j(a/q) \cap \mathfrak{M}_j$ (see e.g \cite[Lemma 8.3]{nat})
$$
\Big\lvert
m_{2^j}(\xi) - 2^{-j} \frac{\mu(q)}{\varphi(q)} \sum_{1 \leq \abs{n} \leq 2^j} e^{2\pi i \theta n}
\Big\rvert
\lesssim j^{-\alpha}
$$
and then, as in the proof of Proposition \ref{prop:1}, we replace the sum by $\Psi_j$.
Also the demonstration of Proposition \ref{prop:2} has to be modified. There, the estimate
for $\xi \not\in \mathfrak{M}_j$ is a direct application of Vinogradov's theorem. In the proof of
Proposition \ref{prop:3} in the place of \eqref{eq:57} we use $L^r$-boundedness of
Hardy--Littlewood maximal function. Eventually, in the proof of Theorem \ref{th:4} we replace
the sum $\sum_{j=0}^k$ with a single term $m_k$.

\begin{bibliography}{discrete}
	\bibliographystyle{amsplain}

\providecommand{\bysame}{\leavevmode\hbox to3em{\hrulefill}\thinspace}
\providecommand{\MR}{\relax\ifhmode\unskip\space\fi MR }
\providecommand{\MRhref}[2]{%
  \href{http://www.ams.org/mathscinet-getitem?mr=#1}{#2}
}
\providecommand{\href}[2]{#2}
\begin{thebibliography}{10}

\bibitem{Bou1}
J.~Bourgain, \emph{{O}n the maximal ergodic theorem for certain subsets of the
  integers}, Israel J. Math. \textbf{61} (1988), 39--72.

\bibitem{Bou2}
\bysame, \emph{{O}n the pointwise ergodic theorem on {$L^p$} for arithmetic
  sets}, Israel J. Math. \textbf{61} (1988), 73--84.

\bibitem{Bou}
\bysame, \emph{Pointwise ergodic theorems for arithmetic sets. {W}ith an
  appendix by the author, {H}arry {F}urstenberg, {Y}itzhak {K}atznelson and
  {D}onald {S}. {O}rnstein.}, Publ. Math.-Paris (1989), no.~69, 5–45.

\bibitem{cot}
M.~Cotlar, \emph{A unified theory of {H}ilbert transforms and ergodic
  theorems}, Rev. Mat. Cuyana \textbf{1} (1955), no.~2, 105--167.

\bibitem{IMSW}
A.~D. Ionescu, E.~M. Stein, A.~Magyar, and S.~Wainger, \emph{Discrete {R}adon
  transforms and applications to ergodic theory}, Acta Math.-Djursholm
  \textbf{198} (2007), no.~2, 231--298.

\bibitem{IW}
A.~D. Ionescu and S.~Wainger, \emph{{$L^p$} boundedness of discrete singular
  {R}adon transforms}, J. Amer. Math. Soc. \textbf{19} (2006), no.~2, 357--383.

\bibitem{MSW}
A.~Magyar, E.~M. Stein, and S.~Wainger, \emph{Discrete analogues in harmonic
  analysis: spherical averages}, Ann. Math. (2002), 189--208.

\bibitem{MT}
M.~Mirek and B.~Trojan, \emph{Discrete for truncated {R}adon transform and
  applications to ergodic theory}, Preprint, 2013.

\bibitem{Na1}
R.~Nair, \emph{On polynomials in primes and {J}. {B}ourgain's circle method
  approach to ergodic theorems}, Ergod. Theor. Dyn. Syst. \textbf{11} (1991),
  485--499.

\bibitem{Na2}
\bysame, \emph{On polynomials in primes and {J}. {B}ourgain's circle method
  approach to ergodic theorems {II}}, Stud. Math. \textbf{105} (1993), no.~3,
  207--233.

\bibitem{nat}
M.~B. Nathanson, \emph{{A}dditive {N}umber {T}heory {T}he {C}lassical {B}ases},
  Graduate Texts in Mathematics, Springer, Princeton, 1996.

\bibitem{O}
D.~M. Oberlin, \emph{Two discrete fractional integrals}, Math. Res. Lett.
  \textbf{8} (2001), 1--6.

\bibitem{P}
L.~B. Pierce, \emph{Discrete analogues in harmonic analysis}, Ph.D. thesis,
  Princeton University, 2009.

\bibitem{P1}
\bysame, \emph{Discrete fractional {R}adon transforms and quadratic forms},
  Duke Math. J. \textbf{161} (2012), no.~1, 69--106.

\bibitem{riv}
N.~M. Rivi\`ere, \emph{Singular integrals and multiplier operators}, Ark. Mat.
  \textbf{9} (1871), no.~2, 243--278.

\bibitem{sieg}
C.~L. Siegel, \emph{{\"U}ber die {C}lassenzahl quadratischer {K}\"orper}, Acta
  Arith. \textbf{1} (1935), 83--96.

\bibitem{SW0}
E.~M. Stein and S.~Wainger, \emph{Discrete analogues in harmonic analysis, {I}:
  {$\ell^2$} estimates for singular {R}adon transforms}, Amer. J. Math.
  \textbf{121} (1999), no.~6, 1291--1336.

\bibitem{SW1}
\bysame, \emph{Discrete analogues in harmonic analysis {II}: fractional
  integration}, J. Anal. Math. \textbf{80} (2000), no.~1, 335--355.

\bibitem{SW2}
\bysame, \emph{Two discrete fractional integral operators revisited}, J. Anal.
  Math. \textbf{87} (2002), no.~1, 451--479.

\bibitem{vau}
R.~C Vaughan, \emph{The {H}ardy--{L}ittlewood {M}ethod}, Cambridge Studies in
  Russian Literature, Cambridge University Press, 1981.

\bibitem{vin}
I.~M. Vinogradov, \emph{The {M}ethod of {T}rigonometrical {S}ums in the
  {T}heory of {N}umbers}, Dover Books on Mathematics Series, Dover
  Publications, 1954.

\bibitem{wal}
A.~Walfisz, \emph{{Z}ur additiven {Z}ahlentheorie. {II}.}, Math. Z. \textbf{40}
  (1936), no.~1, 592--607.

\bibitem{wrl}
M.~Wierdl, \emph{{P}ointwise ergodic theorem along the prime numbers}, Israel
  J. Math. \textbf{64} (1988), no.~3, 315--336.

\end{thebibliography}
\end{bibliography}

\end{document}